\documentclass[a4paper,10pt,reqno, english]{amsart}

\usepackage{amsmath,amssymb,amscd,amsthm,amsfonts}
\usepackage{graphicx,subfigure}
\usepackage{hyperref}
\usepackage{dsfont}

\usepackage{graphicx}

\graphicspath{{Figures/}}

\newtheorem{theorem}{Theorem}[section]
\newtheorem*{theoremp}{Theorem}
\newtheorem{lemma}[theorem]{Lemma}
\newtheorem{claim}[theorem]{Claim}
\newtheorem{corollary}[theorem]{Corollary}

\newtheorem{remark}[theorem]{Remark}
\newtheorem{conjecture}[theorem]{Conjecture}
\newtheorem*{conjecturep}{Conjecture}

\newtheorem*{question}{Question}
\newtheorem{definition}{Definition}

\newcommand{\addone}[1]{\begin{bmatrix}
	#1 \\ 1
\end{bmatrix}}

\def\rr{\mathds{R}}

\DeclareMathOperator{\conv}{conv}
\DeclareMathOperator{\tr}{tr}

\title{Tolerance for colorful Tverberg partitions}

\author[Sarkar]{Sherry Sarkar}
\address{Georgia Tech, North Ave NW, Atlanta, GA 30332} 
\email{ssarkar44@gatech.edu}

\author[Sober\'on]{Pablo Sober\'on}\address{Baruch College, City University of New York, One Bernard Baruch Way, New York, NY 10010,\hfill\break%
\mbox{\hspace{4mm}}United States} 
\email{pablo.soberon-bravo@baruch.cuny.edu}

\thanks{
This research project was done as
part of the 2019 CUNY Combinatorics REU, supported by NSF awards
DMS-1802059 and DMS-1851420.  Sober\'on's research is also supported by PSC-CUNY grant 62639-00-50.}

\begin{document}

\maketitle

\begin{abstract}
Tverberg's theorem bounds the number of points $\mathbb{R}^d$ needed for the existence of a partition into $r$ parts whose convex hulls intersect.  If the points are colored with $N$ colors, we seek partitions where each part has at most one point of each color.  In this manuscript, we bound the number of color classes needed for the existence of partitions where the convex hulls of the parts intersect even after any set of $t$ colors is removed.  We prove asymptotically optimal bounds for $t$ when $r \le d+1$, improve known bounds when $r>d+1$, and give a geometric characterization for the configurations of points for which $t=N-o(N)$.


\end{abstract}

\section{Introduction}

Given a set of points in $\rr^d$, we can study how the convex hulls of its subsets intersect.  These intersections have rich combinatorial properties, which was made clear by Helge Tverberg with his classic theorem.

\begin{theoremp}[Tverberg 1966 \cite{Tverberg:1966tb}]
	Let $r,d$ be positive integers.  Given a set $X$ of $(r-1)(d+1)+1$ points in $\rr^d$, there is a partition of $X$ into $r$ parts whose convex hulls intersect.
\end{theoremp}

Tverberg's theorem is a perfect example of a result at the crossroads of combinatorics, topology, and linear algebra.  Its variations and extensions have provided many fruitful directions of research \cite{Barany:2016vx, Barany:2018fya, DeLoera:2019jb}.  A partition of a set of points where the convex hulls of the parts intersect is called a \textit{Tverberg partition}.  In this manuscript, we focus on how two variations of Tverberg's theorem interact with each other: the colorful version and the version with tolerance.

The colorful version of Tverberg's theorem, conjectured by Imre B\'ar\'any and David Larman, consists of adding combinatorial restrictions to the partitions involved in Tverberg's theorem.  The set of points is divided into color classes.  A partition for which each part has exactly one point of each color is called \textit{colorful}.

\begin{conjecturep}[B\'ar\'any, Larman 1992 \cite{Barany:1992tx}]
	Let $r,d$ be positive integers.  Given $d+1$ color classes $X_1, \ldots, X_{d+1}$ of $r$ points each in $\rr^d$, there exists a colorful partition of their union into $r$ sets whose convex hulls intersect
\end{conjecturep}

 The conjecture has been proved for $d\leq2$ and any $r$ by B\'ar\'any and Larman \cite{Barany:1992tx}, and they showed a proof by Lov\'asz for $r=2$ and any $d$.  Blagojevi\'c, Matschke and Ziegler confirmed the conjecture when $r+1$ is a prime number \cite{Blagojevic:2011vh, Blagojevic:2015wya}, and they proved an optimal version of the B\'ar\'any-Larman conjecture. In several variations of this conjecture, the number and the size of the color classes vary \cite{Zivaljevic:1992vo, Soberon:2015cl, Blagojevic:2014js}. 

The versions with tolerance prove the existence of Tverberg partitions that resist the removal of any sufficiently small set of points.  This line of research started with a result by Larman, who showed that \textit{Given $2d+3$ points in $\rr^d$, there is a partition of them into two sets $A, B$ such that for any point $x$, $\conv(A\setminus \{x\})\cap \conv(B\setminus\{x\}) \neq \emptyset$} \cite{Larman:1972tn}.  The asymptotic behavior as the number of points removed grows to infinity was settled recently.

\begin{theoremp}[Garc\'ia-Col\'in, Raggi, Rold\'an-Pensado 2017 \cite{GarciaColin:2017id}]
	Let $r,d$ be fixed positive integers and let $N$ be a positive integer.  Then, there exists a value $t=N/r-o(N)$ such that for any $N$ points in $\rr^d$ there is a partition of them into $r$ sets that remains a Tverberg partition even if any $t$ points are removed.
\end{theoremp}

There are improved bounds on $t$ for small dimension or small number of parts \cite{Soberon:2012er, Mulzer:2013je, Bereg:2020jy}.  The bounds on the $o(N)$ term have been improved to be polynomial in terms of $N,r,d$, and for $r,d$ fixed it can be replaced by $O(\sqrt{N \ln (N)})$ using the probabilistic method \cite{Soberon:2018gn, Soberon:2019hm}.  Another way to motivate this variation is to think of Tverberg partitions as a game.  First, you make a Tverberg partition.  Then, your enemy sees the partition and removes up to $t$ points.  Your enemy wins if the resulting partition is no longer a Tverberg partition, and you win if it is.  What is the largest value of $t$ you can accept while guaranteeing victory for any set of $N$ points? If the enemy wins, we say that the partition has been broken.

The most surprising aspect of the theorem above is that the leading term that defines $t$ does not depend on the dimension.  It seems that for large values of $N$, the geometry becomes less relevant.  The results of this manuscript show that, for some variations of Tverberg's theorem with tolerance, simple geometric conditions characterize the value of $t$.

There are a few ways to interpret what a ``colorful Tverberg with tolerance'' should mean.  Since the colorful Tverberg theorem imposes conditions on how the partition interacts on the color classes, we impose similar conditions on the removal of points.  We restrict the removal of points to removal of color classes: if our enemy wants to remove a point of color $X_i$, she removes all of $X_i$.  With this condition in mind, the following result is known.  For a positive integer $r$, let $p_r$ the probability that a random permutation of $r$ numbers has at least one fixed point.

\begin{theorem}[Sober\'on 2018 \cite{Soberon:2018gn}]\label{theorem-old-colorful}
Let $r,d$ be fixed, and let $N$ be a positive integer.  There is an integer $t = p_r N - O(\sqrt{N \ln N})$ such that the following is true.  For any $N$ color classes of $r$ points each in $\rr^d$, there is a colorful partition of their union into $r$ sets that remains a Tverberg partition even if any $t$ color classes are removed.
\end{theorem}

The term $p_r$ may seem strange, as $p_r \to 1-1/e$ as $r\to \infty$ but it is not monotone.  Moreover, the smallest value that $p_r$ can take is when $r=2$, giving $p_2 = 1/2$.  This means that the worst case for Theorem \ref{theorem-old-colorful} is when $r=2$.  For most Tverberg-type results, the case $r=2$ is the simplest one.  The second author conjectured that the factor $p_r$ was unnecessary, and that one should be able to achieve $t=N-o(N)$.

\begin{conjecture}[Sober\'on 2018 {\cite[Conjecture 6.1]{Soberon:2018gn}}]\label{conjecture-wrong}
	In Theorem \ref{theorem-old-colorful}, the optimal value of $t$ is $t=N-o(N)$.
\end{conjecture}

In this paper, we disprove Conjecture \ref{conjecture-wrong} by showing that Theorem \ref{theorem-old-colorful} is asymptotically optimal for every $r$ such that $r \le d+1$.

\begin{theorem}
	Let $d \ge 1, d+1 \ge r \ge 2, N \ge 1$ be integers.  There exists a family of $N$ color classes of $r$ points each in $\rr^d$ such that any colorful partition of them fails to be a Tverberg partition after the removal of at most $p_rN$ color classes.
\end{theorem}

We prove the theorem above in Section \ref{sec:simplebound}. For $r=2$, we have a much stronger statement.

\begin{theorem}\label{theorem:optimal-radon}
	Let $d \ge 1, N \ge N' \ge 1$ be integers.    Suppose we have $N$ pairs of points in $\rr^d$(each considered as a color class) such that the maximum number of pairs that any hyperplane splits simultaneously is $N'$.  Then, there exists a colorful partition of the pairs that remains a Tverberg partition even if any $t=N-\frac{N'}{2} - O(\sqrt{N' \ln N})$ pairs are removed.  Moreover, any colorful partition fails to be a Tverberg partition after the removal of at most $N-\frac{N'}{2}$ pairs.    	
\end{theorem}

Therefore, for $r=2$ Conjecture \ref{conjecture-wrong} holds if and only if $N' = o(N)$.  For $d+1 \ge r > 2$, we characterize geometric conditions under which Conjecture \ref{conjecture-wrong} holds in Section \ref{sec:expectation-bounds}. These conditions are given in terms of \textit{perfect split configurations}, defined in Section \ref{sec:simplebound}.

For $r > d+1$, the bound of Theorem \ref{theorem-old-colorful} is not optimal.  In Section \ref{sec:newbounds} we use Helly's theorem in addition to a probabilistic argument to show that the leading term for the tolerance is $q(r,d)N$, where $q(r,d)>p_r$ is a parameter of certain $\{0,1\}$-matrices defined combinatorially.  We fully describe the set of $\{0,1\}$ matrices that can be used to compute this parameter. We also present some open questions. 


\section{Optimal bounds for $r \le d+1$}\label{sec:simplebound}

In this section, we prove an upper bound for tolerance of colorful partitions for $r \le d+1$.  We include it before the preliminaries as it requires no background.  The bound we prove matches the leading term of the upper bound provided by Sober\'on \cite{Soberon:2018gn}. This entails constructing a set of colored points for which any colorful Tverberg partition can be broken by the removal of $p_r N$ points. 

\begin{definition}
Let $N, r$ be positive integers.  Given $N$ color classes of $r$ points of $\rr^d$ each, we say it is a \textbf{perfect split} point configuration if there exist $r$ open half spaces such that
\begin{itemize}
	\item the intersection of all the half-spaces is empty, and
	\item for any color, each half space contains exactly $r-1$ of the $r$ colored points.
\end{itemize}
\end{definition}

\begin{figure}[htp]
    \centerline{\includegraphics[scale=1]{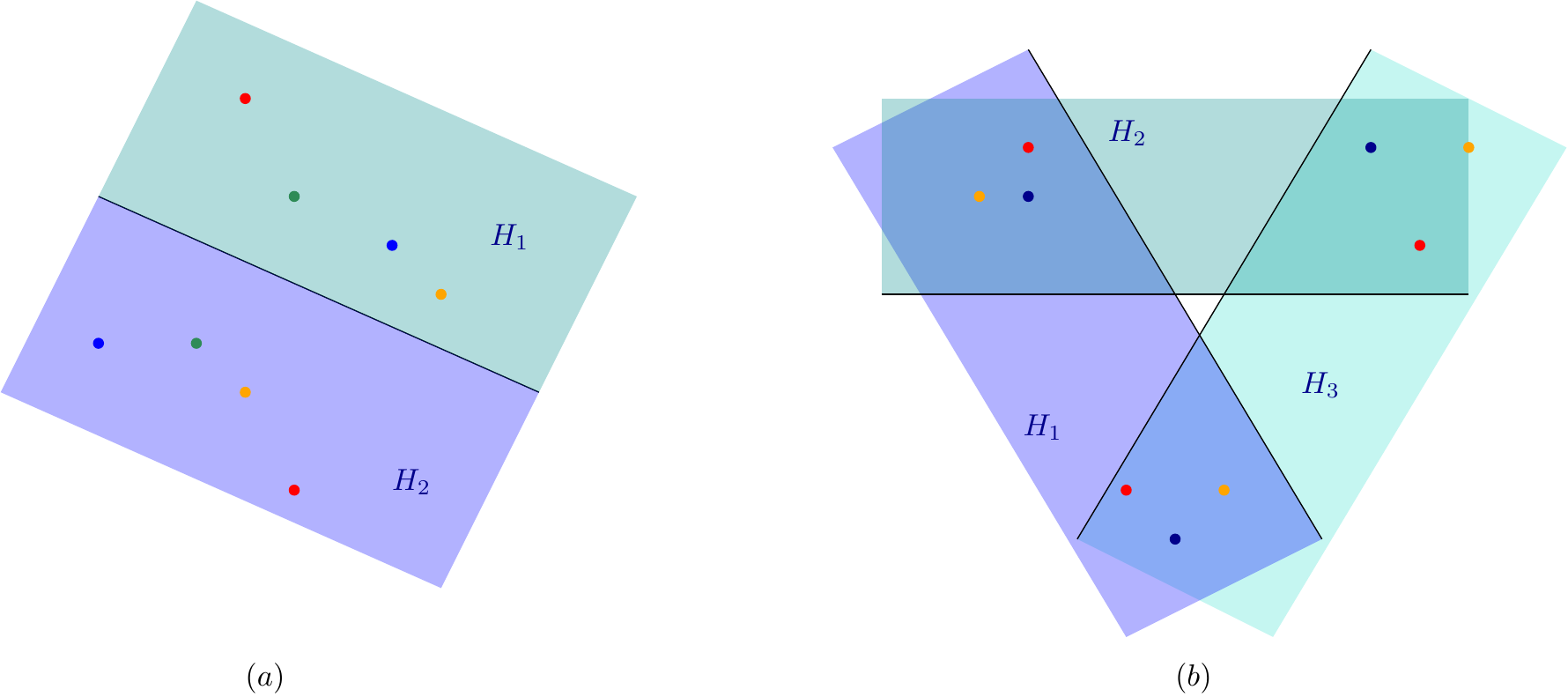}}
    \caption{Examples of two perfect split configurations. (a) Four color classes, each of two points in $\rr^2$. (b)  Three color classes, each of three points in $\rr^2$. }
    \label{fig:splitting-config}
\end{figure}


\begin{theorem}\label{lower_bound}
Suppose we have a perfect split point configuration of $N$ colors and $r$ points of each color in $\rr^d$. For any colorful Tverberg partition of the points into $r$ parts, there exists a choice of $p_rN$ colors one can remove to break the colorful Tverberg partition.
\end{theorem}

\begin{proof}
We take a probabilistic approach. Let $A_1, \ldots, A_r$ be a colorful partition of the set of points. For each color class $X$, we assign an order to its $r$ points based on which part the colored point belongs to. Formally, 
\[ X = \{ x_1, \hdots x_r : x_j \in A_j \mbox{ for }j=1,\ldots, r\}.  \]
Next, we assign random labels $H_1, \hdots, H_r$ to each half-space defining the perfect split configuration. A color class $X$ is removed if there exists an $i$ such that $x_i \in H_i^c$. Note that this guarantees what is left of $A_i$ is contained in $H_i$.  Since the half spaces $H_i$ have empty intersection, the colorful partition is no longer a Tverberg partition. For a given color class $X$, the probability that at least one half space is assigned the same label as the single point of $X$ which it fails to contain is precisely the probability of having at least one fixed point in a permutation, $p_r$. The expected number of colors removed is $p_rN$.  Therefore, there exists some choice of at most $p_rN$ colors we can remove to break the partition. 
\end{proof}

\section{Preliminaries}\label{prelim}\label{sec:prelim}

This paper extends the method taken in Sober\'on's paper \cite{Soberon:2018gn} which is based on Sarkaria's proof of Tverberg's theorem \cite{Sarkaria:1992vt}. Sarkaria's proof relies on a remarkably simple higher dimensional transformation. This transformation in combination with probabilistic techniques are the main ingredients for the proofs of this paper. 

We begin with a key theorem, the colorful Carath\'eodory's theorem.

\begin{theorem}[B\'ar\'any 1982 \cite{Barany:1982va}]
Suppose we have $n + 1$ sets of points $A_1, \hdots, A_{n + 1}$ in $\rr^n$ such that $ 0 \in \conv(A_i)$ for each $i \in [n + 1]$. Then, there exists a traversal $A = \{ a_i  \in A_i : i \in [n + 1] \}$ such that $ 0 \in \conv (A).$
\end{theorem}

We say that a set $X$ captures the origin if $0 \in \conv X$.  We present a sketch of Sarkaria's proof of Tverberg's theorem, which will introduce several constructions relvant in our proofs. This presentation follows the lines of a simplification by B\'ar\'any and Onn \cite{Barany:1995tg}.

\begin{proof}[Proof of Tverberg's theorem] Let $x_0, x_1, \hdots, x_n$ be points in $\rr^d$ where $n = (r - 1)(d + 1)$. Consider $v_1, v_2, \hdots, v_r \in \rr^{r - 1}$ such that $\sum_{i = 1}^r v_i = 0$ is the unique linear dependence up to scalar multipliers. We may also think of the $v_i$ as vertices of an $(r - 1)$-dimensional simplex.  For each $x \in \rr^d$ we can consider $\overline{x} =\addone{x} \in \rr^{d+1}$ the vector consisiting of appending a coordinate $1$ to $x$.
 We consider the following tensor products
\[ v_i \otimes  \overline{x_j} = v_i  \addone{x_j}^T. \]
Note this makes an $(r - 1) \times (d + 1)$ matrix, which we may consider as a vector in $\rr^n$. We construct the sets $A_j$ (corresponding to each point $x_j$)
\[ A_j = \{ v_i \otimes \overline{x_j} \}_{i =1}^r \]
Note that for a given $j$, 
\[ \sum_{i = 1}^r v_i \otimes \overline{x_j} = \left( \sum_{i = 1}^r v_i \right) \otimes \overline{x_j}  = 0, \]
which means that $0 \in \text{conv}A_j$. With this, we may apply Colorful Caratheodory to  each $A_j$ and get some traversal of the $A_j$ such that
\[ \sum_{j = 1}^n \alpha_j v_{f(j)} \otimes \overline{x_j} = 0 \]
where $f : [n] \rightarrow [r]$ indicates which member of $A_j$ is picked in the traversal. Consider a vector $u \in \rr^{r - 1}$ orthogonal to $v_3, \hdots, v_r$ where $\langle u, v_1 \rangle = 1$. Therefore, $\langle u, v_2 \rangle = -1$. If we multiply the above equation on the left by $u^T$, we get
\[ \sum_{j | f(j) = 1} \alpha_j \addone{x_j}^T =  \sum_{j | f(j) = 2} \alpha_j \addone{x_j}^T. \]
The last coordinate implies that the coefficients on both sides have the same sum.  Thus the convex hulls of part 1 and part 2 intersect.  If we call $p$ the point of intersection found above, we may apply a similar argument to see that the convex hulls of the $r$ parts intersect in $p$. 
\end{proof}

\begin{remark}\label{cara-tver}
The proof above provides a parallel between convex hulls of parts intersecting in $\rr^d$ and capturing the origin in $\rr^n$, made explicit in \cite{Arocha:2009ft}. Consider a partition of the set of points $X = \{ x_0, x_1, \hdots, x_n \}$ into $A_1, \hdots, A_r$. Then, 
\[ \bigcap_{i = 1}^r \conv(A_i) \not = \emptyset \]
if and only if
\[ 0 \in \conv(\{ v_{f(j)} \otimes \overline{x_j} \}_{j = 1}^n ) \]
where $f(j)=i$ if and only if $x_j\in A_i$ (i.e $f$ indicates which set each point belongs in). 
\end{remark}

We look more closely at Sarkaria's tensor trick and make several observations. Let $g_i: \rr^d \rightarrow \rr^{(r - 1) \times( d + 1)}$ be 
\[ g_i(x) = v_i \otimes \overline{x}. \]
Consider the $d$ dimensional affine space
\[ U_i = \{ g_i(x) \mid x \in \rr^d \}. \]
We may project each these spaces into $\rr^d$ with the following functions $f_i : \rr^{(r - 1) \times (d + 1)} \rightarrow \rr^d$
\[ f_i(y) = \Pi \left( y^T \frac{v_i}{||v_i||^2} \right) \]
where $\Pi$ denotes the orthogonal projection from $\rr^{d + 1}$ to $\rr^d$ with null space $e^{d + 1}$. More intuitively, $f_i$ is the left inverse of $g_i$ since
\begin{align*} 
f_i(g_i(x)) &= f_i(v_i \otimes \overline{x}) \\
&= \Pi \left(\left( v_i \addone{x}^T \right)^T \frac{v_i}{||v_i||^2} \right) \\
&= \Pi \left( \addone{x} \frac{v_i^T v_i}{||v_i||^2} \right) \\
&= x.
\end{align*}

Throughout the rest of this manuscript, the spaces $U_i$ and the function $f_i$ will refer to those constructed above.

\begin{theorem}\label{high-dim-half-spaces}
Let $r\ge 2,d\ge 1$ be positive integers, and $n = (r-1)(d+1)$. For $i=1,\ldots, r$, consider $v_i$, $U_i$, and $f_i$ as constructed above. Let $H$ be an open half space whose boundary hyperplane contains the origin. Then, for $i=1,\ldots, r$, the set $f_i(H \cap U_i)$ is an open half-space in $\rr^d$ and 
\[ \bigcap_{i = 1}^r f_i(H \cap U_i) = \emptyset. \]
Furthermore, 
\[ \bigcup_{i = 1}^r f_i(\overline{H} \cap U_i) = \rr^d \]
where $\overline{H}$ is the closure of $H$. 
\end{theorem}

\begin{proof}
The half space $H$ in $\rr^{(r - 1) \times (d + 1)}$ may be expressed as 
\[ H = \{ \tr(A^TZ) > 0 \mid A \in \rr^{(r - 1) \times (d + 1)} \} \]
for some matrix $Z$ in $\rr^{(r - 1) \times (d + 1)}$. Therefore, for $i \in [r]$, we have $x = (x_1, \ldots, x_d)^T \in f_i(H \cap U_i)$ if an only 
\begin{align*}
\tr( (v_i \otimes \overline{x})^T Z ) &> 0 \\
\tr \left( \addone{x} v_i^T Z \right) &> 0 \\
\sum_{j = 1}^d x_j \langle v_i, Z_j \rangle &> - \langle v_i, Z_{d + 1} \rangle,  \\
\end{align*}
where $Z_j$ denotes the $j$th column of $Z$ and $\langle \cdot, \cdot \rangle$ denotes the dot product. The last equation clearly defined a half-space in $\rr^d$.  If $x \in f_r(H \cap U_r)$, then 
\begin{align*}
\tr \left( \left(\left(-\sum_{i = 1}^{r - 1} v_i\right) \otimes \overline{x}\right)^T Z \right) &> 0 \\
\tr \left( \addone{x} \left(-\sum_{i = 1}^{r - 1} v_i\right)^T Z \right) &> 0 \\
\sum_{j = 1}^d x_j \left\langle \left(\sum_{i = 1}^{r - 1} v_i\right), Z_j \right\rangle &< - \left\langle \left(\sum_{i = 1}^{r - 1} v_i\right), Z_{d + 1} \right\rangle  \\
\sum_{i = 1}^{r - 1} \sum_{j = 1}^d x_j \langle v_i, Z_j \rangle &< \sum_{i = 1}^{r - 1}  - \langle v_i Z_{d + 1} \rangle
\end{align*}
An element $x \in f_i(H \cap U_i)$ for each $i \in [r-1]$ will satisfy 
\[ \sum_{i = 1}^{r - 1} \sum_{j = 1}^d x_j \langle v_i, Z_j \rangle > \sum_{i = 1}^{r - 1}  - \langle v_i Z_{d + 1} \rangle. \]
This shows us that a point $x \in \rr^d$ cannot be in $\bigcap_{i = 1}^r f_i(H \cap U_i)$. An analogous analysis shows us that
\[ \bigcup_{i = 1}^r f_i( \overline{H} \cap U_i) = \rr^d. \]
\end{proof}

We say that an (open or closed) half space $H \subset \rr^n$ goes through the origin if its boundary contains the origin.

\section{Bounds for $t$ when $r \leq d + 1$}\label{sec:expectation-bounds}

Let $X_1, X_2, \hdots, X_N$ be color classes of $r$ points each in $\rr^d$. We extend Sober\'on's probabilistic method \cite{Soberon:2018gn} to show that there exists a colorful partition of $X_1, \hdots, X_N$ which resists the removal of any $t$ of the color classes. The difference here is that the bounds on the tolerance we establish depend on the geometric constraints on the points. 

\begin{definition}
Let $H_1, \ldots, H_r$ be a family of half spaces in $\rr^d$ and let $X \subset \rr^d$ be a set of $r$ points of $\rr^d$.  We say that $\{H_1, \ldots, H_r\}$ \textbf{can split} the set $X$ if

\begin{itemize}
\item the intersection of all the half-spaces is empty, and
\item the union of any $k$ of the half-spaces contains at least $k$ points of $X$ for $k=1,\ldots, r$.
\end{itemize}

 In particular, we can label the elements of $X$ as $x_1, \ldots, x_r$ such that $x_i \in H_i$ for each $i$.
\end{definition}

Consider the following example.  If we are given $N > rd$, let $y_1, \ldots, y_N$ be points in $\rr^d$ such that no hyperplane contains more than $d$ of them.  Then, we can consider $N$ color classes $X_1, \ldots, X_N$ of $r$ points each in $\rr^d$, such that each $X_i$ is clustered very close to $y_i$.  Any $r$ hyperplanes intersect the convex hulls of at most $rd$ of the $X_i$.  If we have $r$ half spaces with empty intersection, this implies that at least $N-rd$ of the sets $X_i$ must be contained in the complement of at least one of the half-spaces.  In other words, the half spaces cannot split more than $rd$ of the color classes.

\begin{theorem}\label{theorem:geometric-characterization}
	Let $N, r, d$ be positive integers. Suppose we are given $N$ color classes of $r$ points of $\rr^d$ each.   Let $f(N)$ be the largest number of color classes that a family $H_1, \ldots, H_r$ of half spaces in $\rr^d$ can split.  Then:
	\begin{itemize}
		\item There exists a colorful partition of the $Nr$ points such that it remains a Tverberg partition even if any $t$ color classes are removed, for any
		 \[
		 t \le N - (1-p_r)f(N) - \sqrt{\frac{(d+1)(r-1) f(N) \ln (Nr^2)}{2}}-1.
		 \]
		 \item Every colorful partition of the points fails to be a Tverberg partition after the removal of at most
		 \[
		 N-\left(\frac{1}{r!} \right)f(N)
		 \]
		 color classes.
	\end{itemize}
\end{theorem}

We begin by looking at the Sarkaria transformation of our color classes. For each color class $X = \{ x^1, x^2, \hdots, x^r \}$, we consider the following matrix of points

\begin{align*}
\begin{bmatrix}
v_1 \otimes \overline{x^1} & v_2 \otimes \overline{x^1} & \hdots & v_r \otimes \overline{x^1} \\
v_1 \otimes \overline{x^2} & v_2 \otimes \overline{x^2} & \hdots & v_r \otimes \overline{x^2} \\
\vdots & \vdots & \ddots & \vdots \\
v_1 \otimes \overline{x^r} & v_2 \otimes \overline{x^r} & \hdots & v_r \otimes \overline{x^r} 
\end{bmatrix}.
\end{align*}
Note that each row $i$ captures 0 and each column $j$ lies on the $d$ dimensional affine space $U_j$, as defined in Section \ref{prelim}. We call the above matrix an $r$-block. We denote a \textit{colorful choice} of an $r$-block to be a subset of $r$ points in the $r$-block with exactly one point of each row and one point of each column. A \textit{colorful choice} of a collection of $r$-blocks is a subset of points such that, when restricted to each $r$-block, is a colorful choice of the $r$-block.

A colorful choice on a family of $r$-blocks induces a colorful partition in the sets that generated such $r$-blocks, hence why we use the same adjective for both. By Remark \ref{cara-tver}, a colorful choice of a family of $r$-blocks captures the zero vector in $\rr^n$ if its induced partition in $\rr^d$ is a Tverberg partition.

Sober\'on's previous approach \cite{Soberon:2018gn, Soberon:2019hm} takes a random colorful choice over $N$ different $r$-blocks and studies the expected number of colors in a given a half space in $\rr^n$. Then, one can use tail concentration bounds to show that the probability of a colorful partition tolerant to the removal of $t$ colors is non zero.  We show how the geometric conditions on our sets are amenable to these methods.

\begin{proof}[Proof of Theorem \ref{theorem:geometric-characterization}]
	First, let us show the existence of a partition with high tolerance.  Let $n=(r-1)(d+1)$.  For each color color class in $\rr^d$, construct an associated $r$-block of points in $\rr^n$.  We claim that there exists a colorful choice such that every closed half space contains points from at least $t+1$ different $r$-blocks.  This would conclude the proof, as the colorful choice would capture the origin, even after the removal of any $t$ of the $r$-blocks.  We will make this colorful choice randomly.
	
	Let $H$ be a half space through the origin in $\rr^n$ and $B$ an $r$-block and $X$ be its corresponding color class in $\rr^d$. We introduce the random variable $x_B$.  We choose uniformly at random a colorful choice $X_B$ from the elements of $B$. The union of $X_B$ over all $r$-blocks makes a colorful choice for all color classes.  Let $x_B = \chi(X_B \cap H \not = \emptyset)$, the indicator random variable for whether any points of the colorful choice of the colorful choice $X_B$ are in half space $H$. 
	
	Notice that $H^c$, the complement of $H$, is an open half space in $\rr^n$ not containing the zero vector.  By the methods of Section \ref{sec:prelim}, the half spaces $f_1(H^c \cap U_1), \ldots, f_r (H^c \cap U_r)$ in $\rr^d$ do not intersect.
	
	\begin{claim}
		If the half spaces $f_1(H^c \cap U_1), \ldots, f_r (H^c \cap U_r)$ cannot split color $X$, then $\mathbb{E}(x_b) = 1$.
	\end{claim}
	
	\textbf{Proof of claim.} The only way for the expectation to be strictly smaller than $1$ is if there exists a colorful choice $X_B$ such that $X_B \subset H^c$.  Therefore, $H^c \cap U_i$ would contain the the point in the $i$-th column of $X_B$.  This implies that the half spaces $f_1(H^c \cap U_1), \ldots, f_r (H^c \cap U_r)$ can split color $X$.
	
	Therefore, for at least $N-f(N)$ of the $r$-blocks $B$, the expectation of $x_B$ is $1$.  For the rest of the $r$-blocks, we use the lower bound found in Sober\'on's paper \cite{Soberon:2018gn}:
\[ \mathbb{E}(x_B) \geq p_r. \]

Thus, if we sum over all the colors, we get
\[ \sum_B \mathbb{E}(x_B) \geq (N - f(N)) + p_rf(N) = N - (1 - p_r)f(N). \]
At least $N-f(N)$ of the random variables have a fixed value of one.  The rest are independent and all have range in $[0,1]$, so we can apply Hoeffding's inequality to them \cite{Hoe63}.  We obtain 
\[ \mathbb{P}\left( \sum_B x_B \leq N - (1 - p_r)f(N) - \lambda \right) \leq \exp\left(-\frac{2\lambda^2}{f(N)} \right). \]
Set $\lambda > \sqrt{(nf(N)\ln(Nr^2))/2}$. We say that a half space is \textit{bad} if it contains points from fewer than $N - f(N) - \lambda$ different colors. If a half space is bad, the removal of those color classes it contains would separate the convex hull of the colorful choice over all color classes from the origin.   By Remark \ref{cara-tver}, this means the corresponding partition is no longer a Tverberg partition.  Even though there is an infinite number of half spaces $H$ to check, we only need to consider a finite subset of them.  We only distinguish two half spaces if they contain different subsets of the $Nr^2$ points of the $r$-blocks.  We may also assume without loss of generality that the half spaces go through the origin.  By duality, the number of different half spaces is the number of cells in a hyperplane arrangement of $Nr^2$ hyperplanes through the origin in $\rr^n$, which is known to be bounded above by $(Nr^2)^n$.  There are $(Nr^2)^n$ or fewer half spaces to check, so the probability at least one is bad is at most
\[ (Nr^2)^n \exp\left(-\frac{2\lambda^2}{f(N)} \right) < 1. \]
Therefore, there exists a colorful choice for which no hyperplane is bad. This induces the colorful partition in $\rr^d$ we were seeking.

Now, let us show that any partition fails to be Tverberg if we remove enough color classes. The process is similar to that followed in Section \ref{sec:simplebound}. By the definition of $f(N)$, there exist $r$ half spaces that can split $f(N)$ color classes.  We randomly label the half spaces as $H_1, \ldots, H_r$.  For any colorful partition $A_1, \ldots, A_r$ into $r$ parts, we first remove the $N-f(N)$ classes that $H_1, \ldots, H_r$ cannot split.  Then, for each other color class $X$, we number its elements $X=\{x_1, \ldots, x_r\}$ so that $x_i \in H_i$.  From this point on, if we relabel the half spaces we relabel the color classes $X$ accordingly.  We know it is possible to find such an assignment since the half spaces can split $X$.  We remove color $X$ if for some index $i$ we have $x_i \not\in A_i$.  The probability that a color $X$ was not removed is $1/r!$.  After doing this for each color class, what is left in $A_i$ is contained in $H_i$ for $i=1,\ldots,r$. Since the $H_i$ have empty intersection, the partition is no longer a Tverberg partition.  The expected number of color classes removed by this process is $f(N)(1-1/r!)$, so there exists a labeling when at most that number of color classes were removed.  In total, we removed at most $N-f(N)/r!$ color classes.
\end{proof}

In order for Conjecture \ref{conjecture-wrong} to hold, it is sufficient and necessary for $f(N) = o(N)$.  When $r=2$, Theorem \ref{theorem:geometric-characterization} implies Theorem \ref{theorem:optimal-radon}.  For $r \neq 2$, this theorem just shows that $N-t = \Theta(f(N))$.  

\begin{remark}
We make some notes about the differences between a perfect split and a set of $N$ color classes which can be split. For $r=2$, a perfect split with $N$ color classes is equivalent to $N$ color classes which can be split. In general, a perfect split of $N$ colors implies the $N$ colors can be split, but not vice versa. 


\end{remark}


\section{Improved lower bounds on $t$ for $r > d + 1$}\label{sec:newbounds}

The upper bound on the tolerance of Theorem \ref{lower_bound} is achieved at a perfect split configuration. However, when $r > d + 1$, perfect split configurations do not exist. 

\begin{lemma}
Let $N, r, d$ be positive integers.  For $r>d+1$ there does not exist a perfect split point configuration of $N$ color classes of $r$ points each in $\rr^d$.
\end{lemma}

\begin{proof}
Suppose, for the sake of contradiction, that there exist open half-spaces $H_1, \hdots, H_r$ that define a perfect split of our colorful points. Recall that 
\[ \bigcap_{i = 1}^r H_i = \emptyset. \]
By Helly's theorem, there exists a set of $d + 1$ half spaces, $H_{k_1}, \hdots, H_{k_{d+1}}$ such that 
\[ \bigcap_{i = 1}^{d + 1} H_{k_i} = \emptyset. \]
Consider a particular color class $X$. Note that at most $d + 1$ points of $X$ are contained in $\bigcup_{i = 1}^{d + 1} H^{c}_{k_i}$ since each half space $H_i$ contains exactly $r - 1$ points of $X$. Therefore, there exists a point in $\bigcap_{i = 1}^{d + 1} H_{k_i} $ since $r > d + 1$. This is a contradiction. 

\end{proof}

Since the construction of Section \ref{sec:simplebound} fails in this case, it is conceivable that Theorem \ref{theorem-old-colorful} can be improved.  The rest of this section shows that such an improvement is possible.




We can construct a matrix $T$ with $r$ rows and columns. Let $n =(r-1)(d+1)$; for a given closed half space $H$ in $\rr^n$ that contains the origin, we populate matrix $T$ with entries
\[ T(i,j) = \begin{cases}
    1 & \text{if }  u_i \otimes \overline{x_j} \in H \\
    0 & \text{otherwise } 
  \end{cases}\] 

\begin{theorem}\label{new-properties}\label{theorem:new-properties}
Let $r, d$ be positive integers, $n=(r-1)(d+1)$, and $X$ be a set of $r$ points in $\rr^d$.  Let $F$ be the $r$-block in $\rr^n$ induced by $X$, and $H$ a closed half-space space in $\rr^n$ through the origin.  If we construct the matrix $T$ as above, it will must have the following two properties 
\begin{enumerate}
\item Each column $j$ has at least one non zero entry. 
\item There exists a choice of $d + 1$ rows $i_1, \hdots, i_{d + 1}$ such that each column $j$ has at least one non zero entry among $T({i_1, j}), T({i_2, j}), \hdots, T(i_{d + 1}, j)$.
\end{enumerate}
\end{theorem}

\begin{proof}
(1) follows from the fact the each column captures the origin. To see (2), we recall that Theorem \ref{high-dim-half-spaces} asserts 
\[ \bigcap_i f_i(H^c \cap U_i)= \emptyset. \]
For ease of notation, let us denote the half space $f_i(H^c \cap U_i)$ as $H_i$ in $\rr^d$. By Helly's theorem, there exists a choice of $d + 1$ spaces $H_{i_1}, \hdots, H_{i_{d + 1}}$ such that 
\[ \bigcap_{k = 1}^{d + 1} H_{i_k} = \emptyset. \]
Therefore, rows $i_1, \hdots, i_{d + 1}$ in matrix $T$ cannot induce an empty column $j$ in $T$ or else, point $x_j$ is in $\bigcap_{k = 1}^{d + 1} H_{i_k}$. 
\end{proof}

\begin{theorem}\label{theorem:column-probability}
Let $r > d+1$ be a positive integer. Consider an $r \times r$ matrix $T$ with entries $0$ or $1$, satisfying conditions (1) and (2) form Theorem \ref{theorem:new-properties}. The probability that a random permutation $\sigma: [r] \to [r]$ satisfies that $T(i,\sigma(i)) = 1$ for at least one value of $i$ is minimized if each column of $T$ has a single non zero entry, appearing in one of the rows $i_1, \hdots i_{d +1}$ from property (2), and each of the rows $i_1, \ldots, i_{d+1}$ has either exactly $\lfloor r/(d+1) \rfloor$ or exactly $\lceil r/(d+1) \rceil$ entries equal to one.\end{theorem}
\begin{proof}
Without loss of generality, we may assume $i_1, \hdots, i_{d + 1}$ are $1, \hdots, d + 1$.  We may assume that no entry $1$ appears out of rows $1,\ldots, d+1$, or we may replace it by a zero.  It suffices to show that the number of entries equal to $1$ in rows $1$ and in row $2$ differs by at most one.  Suppose in $T$ that row 1 has $k + 1$ non zero entries and row 2 has $k - 1$ non zero entries. Note that a colorful choice of a colored block corresponds to a permutation $\sigma: [r] \rightarrow [r]$ from columns of $T$ to rows of $T$. Let $n(\sigma)$ be the number of non zero entries the permutation $\sigma$ hits in $T$. We will construct a new matrix $T'$ and a corresponding function $m(\sigma)$ such that 
\[ | \{\sigma \mid n(\sigma) = 0 \} | \leq |\{\sigma \mid m(\sigma) = 0 \} |. \] 
This implies that $T'$ has a greater probability that a random permutation does not hit any non zero entry. 

Let $a, b$ be two columns such that $T(1,a) = T(1,b) = 1$ and $T(2,a) = T(2,b) = 0$.  Let $T'$ be the matrix where the $1$ in $T(1,a)$ is moved down to row 2. We next create a bijection $f: S_r \rightarrow S_r$ such that if $n(\sigma) = 0$, then $m(f(\sigma)) = 0$. We describe $f$ with three cases of $\sigma$.
\begin{itemize}
\item If $\sigma(a) \not = 2$: $f(\sigma) = \sigma$. 
\item If $\sigma(a) = 2$ and $1 \neq T(2,\sigma^{-1}(1))$: we define $f(\sigma)(a) = 1$, $f(\sigma)(\sigma^{-1}(1)) = 2$ and $f(\sigma) = \sigma$ on the rest of its values. 
\item If $\sigma(a) = 2$ and $1 = T(2,\sigma^{-1}(1))$: Note that $\sigma(b) \neq 2$, since $\sigma(a) = 2$. Then, we define $f(\sigma)(a) = 1$, $f(\sigma)(\sigma^{-1}(1)) = \sigma(b)$, $f(\sigma)(b) = 2$, and $f(\sigma) = \sigma$ on the rest of its values. 
\end{itemize}
Therefore, the configuration $T'$ which maximizes the probability that a random permutation $\sigma$ does not hit any non zero entry has rows $i_1, \hdots i_{d +1}$ each having $k$ non zero entries. 
\end{proof}

Let us denote by $q(r,d)$ the minimal probability obtained in Theorem \ref{theorem:column-probability}.  Notice that $q(r,d) > p_r$, as the probability of hitting at least one non-zero entry in a square $\{0,1\}$ matrix with exactly one entry $1$ in each column is minimized at permutation matrices.  However, the matrices $T$ we consider for $q(r,d)$ are never permutation matrices.  

If $r$ is a multiple of $d+1$, a standard exclusion-inclusion argument shows that we can compute $q(r,d)$ with the following formula
\[
q(r,d) = \sum_{k=0}^{d+1} (-1)^{k} {{d+1}\choose{k}} \frac{\left(\frac{r}{d+1}\right)^k (r-k)!}{r!}.
\]

Just as with $p_r$, the sequence above converges as $r \to \infty$.  For $r$ not a multiple of $d+1$, one can obtain a formula as above, where some $r/(d+1)$ terms are replaced either by $\lceil r/(d+1) \rceil$ or $\lfloor r/(d+1) \rfloor$.  If we combine Theorem \ref{theorem:new-properties} and Theorem~\ref{theorem:column-probability}, we obtain the following corollary.

\begin{corollary}
Let $r > d + 1$ be positive integers, $n=(r-1)(d+1)$. Let $X$ be a set of $r$ points in $\rr^d$ and $B$ the $r$-block of points in $\rr^n$ constructed from $X$.  Let $H$ be a closed half space in $\rr^n$ that contains the origin. The probability that we have at least one point in a random colorful choice of $B$ is in $H$ is at least $q(r,d)$.
\end{corollary}

This is a bound on the expectation of the random variable $x_B$ as defined in Section \ref{sec:expectation-bounds}.  The exact same proof yields the following theorem, that improves Theorem~\ref{theorem-old-colorful}.

\begin{theorem}
	Let $r,d$ be fixed positive integers such that $r>d+1$, and let $N$ be a positive integer.  There is an integer $t = q(r,d) N - O(\sqrt{N \ln N})$ such that the following is true.  For any $N$ color classes of $r$ points each in $\rr^d$, there is a colorful partition of their union into $r$ sets that remains a Tverberg partition even if any $t$ color classes are removed.
\end{theorem}

In the same manner, for $r>d+1$ we can prove an analogous version of Theorem \ref{theorem:geometric-characterization} where the term $p_r$ is replaced by $q(r,d)$.

\begin{theorem}
	Let $N, r, d$ be positive integers such that $r>d+1$. Suppose we are given $N$ color classes of $r$ points of $\rr^d$ each.   Let $f(N)$ be the largest number of color classes that a family $H_1, \ldots, H_r$ of half spaces in $\rr^d$ with empty intersection can split.  Then:
	\begin{itemize}
		\item There exists a colorful partition of the $Nr$ points such that it remains a Tverberg partition even if any $t$ color classes are removed, for any
		 \[
		 t \le N - (1-q(r,d))f(N) - \sqrt{\frac{(d+1)(r-1) f(N) \ln (Nr^2)}{2}}-1.
		 \]
		 \item Every colorful partition of the points fails to be a Tverberg partition after the removal of at most
		 \[
		 N-\left(\frac{1}{r!} \right)f(N)
		 \]
		 color classes.
	\end{itemize}
\end{theorem}

Unlike in the case of $r \leq d + 1$, we do not have an asymptotically matching upper bound on the colorful tolerance. 
\begin{question}
Let $N, r, d$ be positive integers such that $r>d+1$. Does there exist $N$ color classes of $r$ points each in $\rr^d$ such that for any Tverberg partition of the points, there exists a choice of $q(r,d)N$ color classes one can remove to break the Tverberg partition? 
\end{question}



\newcommand{\etalchar}[1]{$^{#1}$}
\providecommand{\bysame}{\leavevmode\hbox to3em{\hrulefill}\thinspace}
\providecommand{\MR}{\relax\ifhmode\unskip\space\fi MR }
\providecommand{\MRhref}[2]{%
  \href{http://www.ams.org/mathscinet-getitem?mr=#1}{#2}
}
\providecommand{\href}[2]{#2}

\end{document}